\documentclass[letterpaper,12pt]{scrartcl}
\usepackage{amsmath,upref,amssymb,amsthm,amsxtra,exscale}
\usepackage{cite}
\usepackage{enumitem}
\usepackage[T1]{fontenc}
\usepackage{lmodern}
\usepackage{mathtools}

\newcommand{\loc}{\ensuremath{\mathrm{loc}}}

\numberwithin{equation}{section}

\theoremstyle{plain}
\newtheorem{theorem}{Theorem}[section]
\newtheorem{lemma}[theorem]{Lemma}

\theoremstyle{definition}
\newtheorem{definition}[theorem]{Definition}
\newtheorem{remark}[theorem]{Remark}

\newtheorem{example}[theorem]{Example}
\newtheorem*{acknowledgement}{Acknowledgements}

\newcommand{\dN}{\mathbb{N}}

\newcommand{\dR}{\mathbb{R}}

\newcommand{\dZ}{\mathbb{Z}}

\newcommand{\cF}{\mathcal{F}}

\newcommand{\cL}{\mathcal{L}}

\newcommand{\rmd}{\mathrm{d}}

\newcommand{\olB}{\overline{B}}

\DeclareMathOperator{\dist}{dist}

\DeclareMathOperator{\supp}{supp}

\newcommand{\rmloc}{\mathrm{loc}}

\newcommand{\dint}{\,\rmd}
\newcommand{\ssm}{\backslash}
\newcommand{\weakto}{\rightharpoonup}
\newcommand{\id}{\mathrm{id}}

\newcommand{\lr}[3]{#1#3#2}
\newcommand{\xlr}[3]{\left#1#3\right#2}
\newcommand{\biglr}[3]{\bigl#1#3\bigr#2}

\newcommand{\bigglr}[3]{\biggl#1#3\biggr#2}

\newcommand{\abs}[1]{\lr\lvert\rvert{#1}}
\newcommand{\xabs}[1]{\xlr\lvert\rvert{#1}}
\newcommand{\bigabs}[1]{\biglr\lvert\rvert{#1}}

\newcommand{\biggabs}[1]{\bigglr\lvert\rvert{#1}}

\newcommand{\norm}[1]{\lr\lVert\rVert{#1}}

\newcommand{\biggnorm}[1]{\bigglr\lVert\rVert{#1}}

\newcommand{\scp}[1]{\lr\langle\rangle{#1}}

\newcommand{\bigscp}[1]{\biglr\langle\rangle{#1}}

\newcommand{\biggscp}[1]{\bigglr\langle\rangle{#1}}

\newcommand{\fracwithdelims}[4]{\genfrac{#1}{#2}{}{}{#3}{#4}}
\newcommand{\coloneqq}{:=}

\DeclareMathOperator*{\wlim}{w-lim}

\begin{document}

\title{Uniform Continuity and Br\'ezis-Lieb Type Splitting for
  Superposition Operators in Sobolev Space} 

\author{Nils Ackermann\thanks{This research was partially supported by CONACYT grant 237661 and
UNAM-DGAPA-PAPIIT grant IN104315 (Mexico)}\\Instituto de Matem\'aticas\\
Universidad Nacional Aut\'onoma de M\'exico\\
Circuito Exterior, C.U., 04510 M\'exico D.F., M\'exico}
\date{}
\maketitle
\begin{abstract}
  Using concentration-compactness arguments we prove a variant of
  the Br\'ezis-Lieb-Lemma under weaker assumptions on the
  nonlinearity than known before.  An intermediate result on the
  uniform continuity of superposition operators in Sobolev space
  is of independent interest.
\end{abstract}

\section{Introduction}
\label{sec:introduction}

In their seminal paper \cite{MR699419} Br\'ezis and Lieb prove a
result about the decoupling of certain integral expressions,
which has been used extensively in the calculus of variations.
Using concentration compactness arguments in the spirit of
Lions\cite{MR88e:35170,MR778970,MR778974} we prove a variant of
this lemma under weaker assumptions on the nonlinearity than
known before.  To describe a special case of the Br\'ezis-Lieb
lemma, suppose that $\Omega$ is an unbounded domain in $\dR^N$,
$p >1$, $f(t) \coloneqq \abs{t}^p $ for $t\in\dR$, and $(u_n)$ a
bounded sequence in $L^p (\Omega)$ that converges pointwise
almost everywhere to some function $u$.  If one denotes by
$\cF\colon L^p(\Omega)\to L^1(\Omega)$ the superposition operator
induced by $f$, i.e., $\cF(v)(x)\coloneqq f(v(x))$, then the
result in \cite{MR699419} implies that $u\in L^p(\Omega)$ and
\begin{equation}
  \label{eq:26}
  \cF(u_n)-\cF(u_n-u)\to\cF(u)
  \qquad\text{in $L^1(\Omega)$, as $n\to\infty$.}
\end{equation}
The same conclusion is obtained in that paper for more general
functions, imposing conditions that are satisfied for continuous
convex $f$ with $f(0)=0$, and imposing additional conditions on
the sequence $(u_n)$.

A different approach to the decoupling of superposition operators
along sequences of functions rests on certain regularity
assumptions on $f$.  For example, assume that $f\in C^1(\dR)$ satisfies
\begin{equation}
  \label{eq:8}
  \sup_{t\in\dR}\frac{\abs{f'(t)}}{\abs{t}^{p-1}}<\infty.
\end{equation}
Then the proof of \cite[Lemma~8.1]{MR1400007} can easily be
extended to obtain \eqref{eq:26}.  See also the slightly more
general \cite[Lemma~1.3]{MR1885519}, where $f$ is allowed to
depend on $x$ explicitly.

Our aim is to give a decoupling result under a different set of
hypotheses that applies to a much larger class of functions $f$
than considered above, within a certain range of exponents $p$.
In particular, we do not impose any convexity type assumptions on
$f$ as was done in \cite{MR699419}, nor any regularity
assumptions as in \cite{MR1400007,MR1885519} apart from
continuity.  The price we pay for relaxing the hypotheses on $f$
is that we need to restrict the range of allowed growth exponents
$p$ in comparison with \cite{MR699419}, that we need to assume
some type of translation invariance for $\Omega$, and that the
decoupling result only applies to a smaller set of admissible
sequences, namely sequences that converge weakly in
$H^1(\Omega)$.  Nevertheless, the numerous applications in the
Calculus of Variations for PDEs where these extra assumptions are
satisfied justify the new set of hypotheses.

To keep the presentation simple and highlight the main idea, we
only treat the case $\Omega=\dR^N$.  From here on, function
spaces are taken over $\dR^N$ unless otherwise noted.  It would
be possible to consider other domains or superposition operators
between other spaces, and we plan to do so in forthcoming work.
Nevertheless, we do allow a periodic dependency of $f$ on the
space variable.

To explain our results we formalize the notion of decoupling:
\begin{definition}
  \label{def:bl-splitting-along-sequence} 
  Suppose that $X$ and $Y$ are Banach spaces.  Consider a map
  $\cF\colon X\to Y$, a sequence $(u_n)\subseteq X$ and $u\in X$.
  We say that $\cF$ \emph{BL-splits along $(u_n)$ with respect to
    $u$} (BL being an abbreviation for Br\'ezis-Lieb) if
  \begin{equation*}
    \norm{\cF(u_n)- \cF(u_n-u) - \cF(u)}_Y\to 0.
  \end{equation*}
  We say that $\cF$ \emph{almost BL-splits along $(u_n)$ with
    respect to $u$} if, starting with any subsequence of $(u_n)$,
  we can pass to a subsequence such that there is a sequence
  $(v_n)\subseteq X$ such that $\norm{v_n- u}_X\to 0$ and
  \begin{equation*}
    \norm{\cF(u_n)- \cF(u_n-v_n) - \cF(u)}_Y\to 0.
  \end{equation*}

  If $u$ is a limit of $(u_n)$ in some unambiguous sense then we
  frequently omit to mention that (almost) BL-splitting is
  \emph{with respect to $u$}.
\end{definition}

By \cite{MR699419}, the map $f(u)=\abs{u}^p$ induces a map
$\cF\colon L^p\to L^1$ that BL-splits along pointwise a.e.\
converging bounded sequences in $L^p$ with respect to their
pointwise a.e.\ limits.  On the other hand, the technique used to
prove \cite[Lemma~3.2]{MR2088936} (and the related results in
\cite{MR2321894,MR2356423}) yields the following: if
$f\in C(\dR)$ satisfies
\begin{equation}\label{eq:3}
  \sup_{t\in\dR}\frac{\abs{f(t)}}{\abs{t}^p}<\infty  
\end{equation}
then the induced superposition operator $\cF\colon L^p\to L^1$
almost BL-splits along any $L^p_\rmloc$-converging bounded
sequence in $L^p$ with respect to its limit in $L^p_\rmloc$, see
Theorem~\ref{thm:almost-bl-splitting}\ref{item:1} below.  This
result is basically Lion's approach, with a simplifying twist.
If in addition $\cF$ is uniformly continuous on bounded subsets
of $L^p$ then it is easy to see that it BL-splits along any
$L^p_\rmloc$-converging bounded sequence in $L^p$ with respect to
its limit in $L^p_\rmloc$, see \cite[Lemma~6.3]{MR2216902}.  For
example, this holds true if \eqref{eq:8} is satisfied.

We illustrate the distinction between BL-splitting and almost
BL-splitting by the following examples:
\begin{example}
  \label{exa:oscillation} If $p>1$ and if either
  $f(t)\coloneqq\cos(\pi t)\abs{t}^p$ or
  $f(t)\coloneqq\cos(\pi/t)\abs{t}^p$ then there is a bounded
  sequence $(u_n)$ in $L^p$ that converges in $L^p_\rmloc$ and
  pointwise a.e.\ to a function $u$ such that the induced
  continuous superposition operator $\cF\coloneqq L^p\to L^1$
  does not BL-split along any subsequence of $(u_n)$ with respect
  to $u$.  On the other hand, $\cF$ almost BL-splits along any
  $L^p_\rmloc$-converging bounded sequence in $L^p$ with respect
  to its limit in $L^p_\rmloc$.  Hence $\cF$ is not uniformly
  continuous on bounded subsets of $L^p$ and neither the general
  conditions used in \cite{MR699419} nor \eqref{eq:8} are
  satisfied for $f$ in these examples.
\end{example}

The sequences mentioned in the example are provided in
Section~\ref{sec:constr-exampl} below.

Our main interest is to avoid condition \eqref{eq:8}, or any
other conditions on $f$ that ensure uniform continuity on bounded
subsets of $L^p$ (e.g., a local H\"older condition, together with
an appropriate growth bound on the H\"older constants on bounded
intervals).  Our result below states that it is sufficient to
restrict to bounded subsets of $H^1$ instead.

In this context we now formulate our main theorem, in a slightly
more general setting than what we considered above.
A function $f\colon \dR^N \times \dR\to\dR$ is a
\emph{Caratheodory function} if $f$ is measurable and if
$f(x,\cdot)$ is continuous for almost every $x\in\dR^N$.  The
induced superposition operator on functions $u\colon \dR^N\to\dR$
is then given by $\cF(u)(x) \coloneqq f(x,u(x))$.  If $A$ is a
real invertible $N\times N$-matrix then $f$ is said to be
$A$-periodic in its first argument if $f(x+Ak,t)=f(x,t)$ for all
$x\in\dR^N$, $k\in\dZ^N$, and $t\in\dR$.

Denote by $2^*\coloneqq2N/(N-2)$ if $N\ge3$ and
$2^*\coloneqq\infty$ if $N=1$ or $N=2$ the critical Sobolev
exponent for $H^1$.  Recall the continuous and compact embedding
of the Sobolev space $H^1(U)$ in $L^p(U)$ for
$p\in[2,2^*)$ if $U\subseteq\dR^N$ is a bounded domain.

\begin{theorem}
  \label{thm:main} Consider $\mu>0$, $\nu\ge1$, and $C_0>0$, such
  that $p\coloneqq\mu\nu\in (2,2^*)$.  Suppose that
  $f\colon\dR^N\times\dR \to\dR$ is a Caratheodory function that
  satisfies
  \begin{equation}
    \label{eq:1}
    \abs{f(x,t)}\le C_0\abs{t}^\mu
    \qquad\text{for all } x\in\dR^N,\ t\in\dR,
  \end{equation}
  and which is $A$-periodic in its first argument, for some
  invertible matrix $A\in\dR^{N\times N}$.  Denote by $\cF \colon
  L^p\to L^\nu$ the continuous superposition
  operator induced by $f$.  Then $\cF$ is uniformly continuous on
  bounded subsets of $H^1$ with respect to the
  $L^p$-$L^\nu$-norms and hence also with respect to the
  $H^1$-$L^\nu$-norms.  Moreover, $\cF\colon H^1\to
  L^\nu$ BL-splits along weakly convergent sequences in
  $H^1$ with respect to their weak limit.
\end{theorem}

Our proof of Theorem~\ref{thm:main} has similarities with the
proof of \cite[Theorem~3.1]{MR2294665} but involves an
intermediate cut-off step in the proof of
Theorem~\ref{thm:almost-bl-splitting}.  Essentially, we first
prove almost BL-splitting of $\cF$ along weakly converging
sequences in $H^1$ with respect to their weak limit, using the
concentration function and the compactness of the Sobolev
embedding $H^1(U)\hookrightarrow L^p(U)$, for $p\in[2,2^*)$ and
for a bounded domain $U$.  Then we collect the possible mass loss
at infinity along subsequences with the help of Lions' Vanishing
Lemma, employing the assumption $p>2$.

\begin{remark}\label{rem:counterexamples}
  Theorem~\ref{thm:main} applies in particular to the functions
  considered in Example~\ref{exa:oscillation} when $\nu=1$ and
  $\mu=p\in(2,2^*)$.  On the other hand, for
  $f(t)\coloneqq\cos(\pi/t)t^2$ there is a sequence in $H^1$ that
  converges weakly but that possesses no subsequence along which
  $f\colon H^1\to L^1$ BL-splits with respect to the weak limit.
  The same is true for $f(t)\coloneqq\cos(\pi t)t^{2^*}$.  In
  this sense, Theorem~\ref{thm:main} is optimal, that is, it
  cannot be extended in this generality to include the cases
  $p=\mu\nu=2$ and $p=\mu\nu=2^*$.  The existence of these
  counterexamples is proved in Section~\ref{sec:constr-exampl}.

  Of course, by Sobolev's embedding theorem, a map $L^p\to L^\nu$
  that BL-splits along $L^p_\loc$-converging bounded sequences in
  $L^p$ with respect to their limits in $L^p_\loc$ also BL-splits
  in $H^1$ along weakly convergent sequences with respect to
  their weak limits.  Therefore
  Theorem~\ref{thm:almost-bl-splitting}\ref{item:1}, together
  with \eqref{eq:8} (or the weaker H\"older condition with growth
  bound), yields BL-splitting maps along weakly convergent
  sequences in $H^1$ with respect to weak limits even for $p=2$
  and $p=2^*$.
\end{remark}
\begin{remark}
  The result also holds true in a slightly restricted sense for
  functions $f$ that are sums of functions as in
  Theorem~\ref{thm:main}, i.e., functions that satisfy merely
  \begin{equation*}
    \abs{f(x,t)}\le C_0(\abs{t}^{\mu_1}+\abs{t}^{\mu_2})
    \qquad\text{for all } x\in\dR^N,\ t\in\dR,
  \end{equation*}
  where $\mu_i\nu\in(2,2^*)$ for $i=1,2$.  In that case,
  $\cF\colon H^1\to L^\nu$ is uniformly continuous
  on bounded subsets of $H^1$ with respect to the
  $H^1$-$L^\nu$ norms, and $\cF$ BL-splits along weakly
  convergent sequences in $H^1$ with respect to their weak
  limits.
\end{remark}

\begin{remark}
  The uniform continuity of operators $\cF$ on bounded subsets of
  $H^1$ has been used, for example, in the proof of
  \cite[Lemma~3.4]{MR2294665}.  Nevertheless, we are not aware of
  a published proof of this fact, which is nontrivial in the
  generality stated in Theorem~\ref{thm:main}.  Note that the
  uniform continuity of $\cF\colon H^1(U)\to L^\nu(U)$ on bounded
  subsets of $H^1(U)$ is trivial if $U$ is bounded, by the
  compact Sobolev embedding $H^1(U)\subseteq L^p(U)$.
\end{remark}

We now discuss additional aspects and applications of the results
presented above.  To this end we return to a simple setting on
$\dR^N$.  Suppose that $f\in C(\dR)$ satisfies \eqref{eq:3} with
$p\in(2,2^*)$ and consider the functional $\Phi\colon H^1\to\dR$
given by
\begin{equation*}
  \Phi(u)\coloneqq\int_{\dR^N}f(u).
\end{equation*}
To prove the existence of a minimizer in typical variational
problems involving $\Phi$, Lions~\cite{MR778970,MR778974}
introduces the concentration-compactness principle.  It is a tool
to exclude the possibility of vanishing and of dichotomy along a
minimizing sequence $(u_n)$, in order to obtain compactness of
the sequence.  Here we are only concerned with dichotomy.  In
this case, the sequence $(u_n)$ is approximated by
$(u^1_n+u^2_n)$, where
$\dist(\supp(u^1_n), \supp(u^2_n))\to\infty$.  For \emph{local}
functionals like $\Phi$ it then follows easily that $\Phi(u_n)$
is approximated by $\Phi(u^1_n)+\Phi(u^2_n)$, a fact that yields,
together with a hypothesis about energy levels, a contradiction.
Clearly, the same can be achieved if $\Phi$ BL-splits along
$(u_n)$ in a suitable way.  Before our Theorem~\ref{thm:main},
Lions' approach to concentration compactness was more general, in
that, besides continuity and appropriate growth bounds, no extra
regularity hypotheses need to be placed on $f$.  On the other
hand, the arguments are more involved than when using
BL-splitting because one has to insert cut-off functions to
obtain sequences $u^1_n$ and $u^2_n$ with disjoint supports.  As
a consequence, it is difficult to give a purely functional
(abstract) presentation of Lions' approach.

To explain the advantage of an abstract presentation using
BL-splitting, we note that to treat nonlocal functionals of
convolution type, e.g.,
\begin{equation*}
  \Psi(u)\coloneqq\int_{\dR^N}(f*h(u))h(u),
\end{equation*}
the property of disjoint supports is not as effective anymore.
In the convolution, the supports get ``smeared out'' and one has
to control the interaction with more involved estimates, see
page~123 of~\cite{MR778970}.  This is aggravated when one also
has to consider the decoupling of \emph{derivatives} of $\Psi$.
We have shown in \cite{MR2216902} that using BL-splitting is
effective in situations involving nonlocal functionals.
Moreover, BL-splitting even survives certain nonlocal operations,
like the saddle point reduction,
see~\cite[Theorem~5.1]{MR2216902}.

For particular cases there are other approaches to avoid
conditions on $f$ besides continuity and growth bounds.  We
reformulate and simplify the following cited results slightly to
adapt them to our setting and notation.  In \cite{MR2151860} we
proved, for $f\in C(\dR)$ satisfying \eqref{eq:1} with
$\mu\coloneqq p-1$, and setting $\nu\coloneqq p/(p-1)$, that the
map $\Gamma\colon H^1\to H^{-1}$, given by
\begin{equation*}
  \Gamma(u)v\coloneqq\int_{\dR^N}f(u)v,
\end{equation*}
BL-splits along a weakly convergent sequence if the weak limit is
a function tending to $0$ as $\abs{x}\to\infty$.  Another result
was given in \cite[Lemma~7.2]{MR2491945}, when $f\in C(\dR)$
satisfies \eqref{eq:1} with $\mu\coloneqq p-2$ and
$\nu\coloneqq p/(p-2)$: The map
$\Lambda\colon H^1\to\cL_2(H^1,\dR)$ (here $\cL_2(H^1,\dR)$
denotes the space of bounded bilinear maps from $H^1$ into
$\dR$), given by
\begin{equation*}
  \Lambda(u)[v,w]\coloneqq\int_{\dR^N}f(u)vw,
\end{equation*}
is uniformly continuous on bounded subsets of $H^1$.  Together
with the almost BL-splitting of $\Lambda$ given by
Theorem~\ref{thm:almost-bl-splitting} below this yields
BL-splitting for $\Lambda$ along weakly convergent sequences.
Note that the idea of the proof of the latter result does not
apply for the maps $\Phi$ and $\Gamma$ defined above (under the
respective growth bounds on $f$).  In both cases our result here
is stronger, since we show uniform continuity and BL-splitting
into the spaces $L^\nu$, which are continuously embedded in
$H^{-1}$ and $\cL_2(H^1,\dR)$, respectively.

A different application of Theorem~\ref{thm:main}, that is
independent of variational methods, is the general study of maps
that are uniformly continuous on a subset of an infinite
dimensional Hilbert space.  These play a role in infinite
dimensional potential theory \cite{MR0227747,MR3109657} or, more
generally, in the theory of stochastic equations in infinite
dimensions \cite{MR1724248,MR3236753,MR1985790}.

The paper is structured as follows.  In
Section~\ref{sec:proof-theorem} we treat almost BL-splitting of
$\cF$ along bounded sequences in $L^p$ that converge in
$L^p_\rmloc$, and along weakly convergent sequences in $H^1$.  In
Section~\ref{sec:uniform-continuity} we prove the uniform
continuity of $\cF$ on bounded subsets of $H^1$ and BL-splitting
of $\cF$ along weakly convergent sequences.  In
Section~\ref{sec:constr-exampl} we prove the claims made in
Example~\ref{exa:oscillation}.

\section{Almost BL-Splitting}
\label{sec:proof-theorem}

In this section we prove a result on the almost BL-splitting of
superposition operators in $L^p$ along bounded sequences that
converge in $L^p_\rmloc$, and in $H^1$ along weakly convergent
sequences.  This is a variation on Lions' approach in
\cite{MR778970}.  Note that here the periodicity assumption in
$x$ is not needed.

If $r\in[1,\infty]$ then denote by $\abs{\,\cdot\,}_r$ the norm
of $L^r$.  

\begin{theorem}
  \label{thm:almost-bl-splitting} Consider $\mu>0$, $\nu\ge1$,
  and $C_0>0$, such that $p\coloneqq\mu\nu\ge 1$.  Suppose that
  $f\colon\dR^N\times\dR \to\dR$ is a Caratheodory function that
  satisfies \eqref{eq:1}.  Denote by $\cF$ the superposition
  operator on real functions induced by $f$.
  \begin{enumerate}[label=\textup{(\alph*)}]
  \item \label{item:1} If $(u_n)\subseteq L^p$ is bounded
    and converges in $L^p_\loc$ to a function $u$, then
    $u\in L^p$ and $\cF \colon L^p\to L^\nu$
    almost BL-splits along $(u_n)$ with respect to $u$.
  \item \label{item:2} If $p\in[2,2^*)$ and $u_n\weakto u$ in
    $H^1$ then $\cF\colon H^1\to L^\nu$
    almost BL-splits along $(u_n)$ with respect to $u$.
  \item \label{item:3} In \ref{item:2}, if in addition
    $(\bar{u}_n)\subseteq H^1$ converges weakly and
    $\abs{u_n-\bar{u}_n}_p\to0$ as $n\to\infty$ then
    $\bar{u}_n\weakto u$ in $H^1$ and $\cF$ almost
    BL-splits along $(u_n)$ and $(\bar{u}_n)$ with respect to
    $u$, preserving subsequences and the auxiliary sequence
    $(v_n)$ in the following sense: for any subsequence $n_k$
    there is a subsequence $n_{k_\ell}$ and $(v_\ell)$ such that
    $v_\ell\to u$ in $H^1$ and, writing
    $u_\ell\coloneqq u_{n_{k_\ell}}$ and
    $\bar{u}_\ell\coloneqq \bar{u}_{n_{k_\ell}}$ we have
    \begin{align*}
      \cF(u_\ell)-\cF(u_\ell-v_\ell)&\to\cF(u)\\
      \shortintertext{and}
      \cF(\bar{u}_\ell)-\cF(\bar{u}_\ell-v_\ell)&\to\cF(u).
    \end{align*}
  \end{enumerate}
\end{theorem}

For the proof, let $B_R$ denote, for $R>0$, the open ball in
$\dR^N$ with center $0$ and radius $R$.

\begin{proof}
  \textbf{\ref{item:1}:} From \eqref{eq:1} and from the theory of
  superposition operators~\cite{MR1066204} it follows that
  $\cF\colon L^p(U)\to L^\nu(U)$ is continuous for any open
  subset $U$ of $\dR^N$.  For $n\in\dN$ define
  $Q_n\colon[0,\infty)\to[0,\infty)$ by
  \begin{equation*}
    Q_n(R):=\int_{B_R}\abs{u_n}^{p}. 
  \end{equation*}
  The functions $Q_n$ are uniformly bounded and nondecreasing.
  We may assume that $(Q_n)$ converges pointwise almost
  everywhere to a bounded nondecreasing function $Q$
  \cite{MR778970}.  It is easy to build a sequence
  $R_n\to\infty$ such that for every $\varepsilon>0$ there is $R>0$,
  arbitrarily large, with
  \begin{equation*}
    \limsup_{n\to\infty}(Q_n(R_n)-Q_n(R))\le\varepsilon. 
  \end{equation*}
  Hence 
  \begin{equation}\label{eq:10}
    \forall\varepsilon>0\ \exists R>0\colon
    \limsup_{n\to\infty}\int_{B_{R_n}\ssm
      B_R}\abs{u_n}^p\le\varepsilon
    \qquad\text{and}\qquad
    \int_{\dR^N\ssm B_R}\abs{u}^p\le\varepsilon.
  \end{equation}
  
  Consider a smooth cut off function $\eta \colon[0,\infty) \to
  [0,1]$ such that $\eta\equiv1$ on $[0,1]$ and $\eta\equiv0$ on
  $[2,\infty)$.  Set $v_n(x)\coloneqq\eta(2\abs{x}/R_n)u(x)$.
  Then
  \begin{equation}
    \label{eq:12}
    \lim_{n\to\infty}v_n= u\qquad\text{in $L^p$.}
  \end{equation}
  From the continuity of $\cF$ on $L^{p}(B_R)$, $v_n=u$ on $B_R$,
  $\lim_{n\to\infty} u_n = u$ in $L^p(B_R)$, and $f(x,0)=0$ for
  a.e.\ $x\in\dR^N$ we obtain
  \begin{multline*}
    \lim_{n\to\infty}\int_{B_R}
    \bigabs{f(x,u_n)-f(x,u_n-v_n)-f(x,v_n)}^{\nu}\dint x\\
    =\lim_{n\to\infty}\int_{B_R}
    \bigabs{f(x,u_n)-f(x,u_n-u)-f(x,u)}^{\nu}\dint x=0. 
  \end{multline*}
  Since $v_n\equiv0$ in $\dR^N\ssm B_{R_n}$, this in turn
  yields for any $\varepsilon>0$ and $R$ chosen accordingly, as in
  \eqref{eq:10},
  \begin{multline*}
    \limsup_{n\to\infty}\int_{\dR^N}
    \abs{f(x,u_n)-f(x,u_n-v_n)-f(x,v_n)}^{\nu}\dint x\\
    \begin{aligned}
      &= \limsup_{n\to\infty}\int_{B_{R_n}\ssm B_R}
      \abs{f(x,u_n)-f(x,u_n-v_n)-f(x,v_n)}^{\nu}\dint x\\
      &\le C\limsup_{n\to\infty}\int_{B_{R_n}\ssm B_R}
      (\abs{u_n}^\mu+\abs{u_n-v_n}^\mu+\abs{v_n}^\mu)^{\nu}\\
      &\le C\limsup_{n\to\infty}\int_{B_{R_n}\ssm B_R}
      (\abs{u_n}^{p}+\abs{u}^{p})\\
      &\le C\varepsilon,
    \end{aligned}
  \end{multline*}
  where $C$ is independent of $\varepsilon$.  Letting
  $\varepsilon$ tend to $0$ and using \eqref{eq:12} we obtain
  \begin{equation*}
    \lim_{n\to\infty}\bigabs{\cF(u_n)-\cF(u_n-v_n)-\cF(u)}_\nu= 0.
  \end{equation*}

  \textbf{\ref{item:2}:} The continuous embedding
  $H^1\hookrightarrow L^p$ implies that $(u_n)$ is bounded in
  $L^p$, and the compact embedding $H^1(U)\hookrightarrow L^p(U)$
  for bounded $U$ implies that $u_n\to u$ in $L^p_\rmloc$.
  Defining $v_n$ as in \ref{item:1} we therefore obtain that
  \begin{equation}
    \label{eq:13}
    v_n\to u\qquad\text{in } H^1,
  \end{equation}
  and $\cF$ almost BL-splits along $(u_n)$ with respect to $u$ by \ref{item:1}.

  \textbf{\ref{item:3}:} Since $\abs{u_n-\bar{u}_n}_p\to0$ and
  $u_n\to v$ in $L^p_\rmloc$ it follows that $\bar u_n\weakto v$
  in $H^1$.  Taking $R$ large enough, \eqref{eq:10} also holds
  true if we replace $u_n$ by $\bar{u}_n$.  Therefore, after
  passing to a subsequence for $(u_n)$, and using the same
  subsequence for $(\bar u_n)$, we obtain
  \begin{equation*}
    \lim_{n\to\infty}\bigabs{\mathcal{F}(\bar{u}_n)
    -\mathcal{F}(\bar{u}_n-v_n)-\mathcal{F}(u)}_\nu=0.\qedhere
  \end{equation*}
\end{proof}

\section{Uniform Continuity}
\label{sec:uniform-continuity}

Here we prove uniform continuity on bounded subsets of $H^1$,
making use of the periodicity of $f$ in $x$.  As a consequence,
we also obtain BL-splitting along weakly convergent sequences in
$H^1$.

For simplicity we will only prove the case $A=I$ (the identity
transformation).  The general case follows in an analogous
manner.  Denote the respective translation action of the additive
group $\dZ^N$ on functions $u\colon\dR^N\to\dR$ by
\begin{equation*}
  (a\star u)(x)\coloneqq u(x-a),\qquad a\in\dZ^N,\ x\in\dR^N.
\end{equation*}
Let $\scp{\cdot,\cdot}$ denote the standard scalar product in
$H^1$, defined by
\begin{equation*}
  \scp{u,v}\coloneqq\int_{\dR^N}(\nabla u\cdot\nabla v+uv),
\end{equation*}
and let $\norm{\,\cdot\,}$ denote the associated norm.  Also
denote by $\wlim$ the weak limit of a weakly convergent sequence.

We first recall a functional consequence of Lions' Vanishing
Lemma, \cite[Lemma~I.1.]{MR778974}.
\begin{lemma}
  \label{lem:lions} Suppose for a sequence $(u_n)\subseteq H^1$
  that $a_n\star u_n\weakto 0$ in $H^1$ for every sequence
  $(a_n)\subseteq\dZ^N$.  Then $u_n\to0$ in $L^p$ for all
  $p\in(2,2^*)$.
\end{lemma}
\begin{proof}
  Note first that $(u_n)$ is bounded in $H^1$ since $u_n\weakto0$
  in $H^1$.  We claim that
  \begin{equation}
    \label{eq:2}
    \sup_{y\in\dR^N}\int_{y+B_1}\abs{u_n}^2\to0
    \qquad\text{as }n\to\infty.
  \end{equation}
  If the claim were not true there would exist $\varepsilon>0$
  and a sequence $(y_n)\subseteq\dR^N$ such that, after passing
  to a subsequence of $(u_n)$,
  \begin{equation*}
    \int_{y_n+B_1}\abs{u_n}^2\ge\varepsilon.
  \end{equation*}
  Pick $(a_n)\subseteq\dZ^N$ such that $\abs{a_n+y_n}_\infty<1$
  for all $n$.  With $R\coloneqq\sqrt{N}+1$ it follows that
  $a_n+y_n+B_1\subseteq B_R$ and hence
  \begin{equation*}
    \int_{B_R}\abs{a_n\star u_n}^2\ge\varepsilon
  \end{equation*}
  for all $n$.  We reach a contradiction since $a_n\star
  u_n\weakto0$ in $H^1$ and hence $a_n\star u_n\to0$ in
  $L^2(B_R)$ by the theorem of Rellich and Kondrakov.  Therefore
  \eqref{eq:2} holds true.

  The claim of the theorem now follows from
  \cite[Lemma~I.1.]{MR778974} with $p=q=2$.  Compare also
  with \cite[Lemma~3.3]{MR2294665}.
\end{proof}
\begin{proof}[Proof of Theorem~\ref{thm:main}]
  We start by proving the uniform continuity.  Let
  $(u^0_{i,n})_{n\in\dN_0}$ be bounded sequences in $H^1$ for
  $i=1,2$ and set $C_1\coloneqq\max_{i=1,2}
  \limsup_{n\to\infty}\norm{u^0_{i,n}}$.  Suppose for a
  contradiction that
  \begin{equation}
    \abs{u^0_{1,n}-u^0_{2,n}}_p\to0 
    \qquad\text{as } 
    n\to\infty,\label{eq:4}
  \end{equation}
  and that there is $ C_2>0$ such that
  \begin{equation}
    \abs{\cF(u^0_{1,n})-\cF(u^0_{2,n})}_\nu\ge C_2
    \qquad\text{for all }n.\label{eq:6}
  \end{equation}
  
  Successively we will define infinitely many sequences
  $(a^k_n)_n\subseteq\dZ^N$ and $(u^k_{i,n})_n\subseteq H^1$,
  $i=1,2$, indexed by $k\in\dN_0$ and strictly increasing
  functions $\varphi_k\colon\dN\to\dN$ with the following
  properties:
  \begin{align}
    \max_{i=1,2}\limsup_{n\to\infty}\norm{u^k_{i,n}}&\le C_1,\label{eq:5}\\
    \lim_{n\to\infty}\abs{u^k_{1,n}-u^k_{2,n}}_p&=0,\label{eq:15}\\[.5ex]
    \liminf_{n\to\infty}\abs{\cF(u^{k}_{1,n})-\cF(u^{k}_{2,n})}_\nu
    &\ge C_2,\label{eq:11}\\
    \wlim_{n\to\infty}\xlr(){-a^\ell_{\psi^{k-1}_\ell(n)}}\star u^k_{i,n}&= 0
    &&\text{in $H^1$, if  $0\le\ell<k$, for
      $i=1,2$,}\label{eq:20}\\
    \shortintertext{and}
    \lim_{n\to\infty}\abs{a^m_{\psi^\ell_m(n)}-a^\ell_n}&=\infty
    &&\text{if }0\le m<\ell<k.\label{eq:29}
  \end{align}
  Here
  \begin{align*}
    \psi^k_\ell
    &\coloneqq
    \varphi_{\ell+1}\circ\varphi_{\ell+2}\circ\dots\circ\varphi_k
    &&\text{if } \ell=-1,0,1,\dots,k-1\\
    \psi^k_k
    &\coloneqq\id_{\dN}.
  \end{align*}

  We need to say something about the extraction of subsequences.
  In order to obtain $\varphi_k$, $(a^{k}_n)_n$, and
  $(u^{k+1}_{i,n})_n$ from $(u^k_{i,n})$, we first pass to a
  subsequence $(u^k_{i,\varphi_k(n)})_n$ of $(u^k_{i,n})_n$ and
  then use its terms in the construction.  Once the new sequences
  $(a^{k}_n)_n$ and $(u^{k+1}_{i,n})_n$ are built we may remove a
  finite number of terms at their start, modifying $\varphi_k$
  accordingly, with the goal of obtaining additional properties.
  Beginning with the following iteration there are no more
  retrospective changes to the sequences already built.  This is
  to assure a well defined infinite sequence of sequences, from
  which eventually we take the diagonal sequence.  In this
  setting it seems clearer to make the selection of subsequences
  explicit, contrary to what is usually done when using
  concentration compactness methods
  \cite{MR88e:35170,MR778970,MR778974} or when proving a
  variational splitting lemma.

  For $k=0$ the properties \eqref{eq:5}--~\eqref{eq:29} are
  fulfilled by the definition of $C_1$ and by \eqref{eq:4} and
  \eqref{eq:6}.  Assume now that \eqref{eq:5}--\eqref{eq:29} hold
  for some $k\in\dN_0$.  Denote by $W_k$ the set of $v\in H^1$
  such that there are a sequence $(a_n)\subseteq\dZ^N$ and a
  subsequence of $(u^k_{1,n})$ with
  $\wlim_{n\to\infty}a_n\star u^k_{1,n}=v$ in $H^1$.
  
  If $\wlim_{n\to\infty}a_n\star u^k_{1,n}=0$ in $H^1$ were true
  for all sequences $(a_n)\subseteq\dZ^N$, by
  Lemma~\ref{lem:lions} it would follow that $\lim_{n\to\infty}
  u^k_{1,n}=0$ in $L^p$.  Equation~\eqref{eq:15} and the
  continuity of $\cF$ on $L^p$ would lead to a contradiction with
  \eqref{eq:11}.  Therefore
  \begin{equation*}
    q_k\coloneqq \sup_{v\in W_k} \norm{v}\in(0,C_1].
  \end{equation*}

  Pick $v^k\in W_k$ such that
  \begin{equation}
    \label{eq:7}
    \norm{v^k}\ge \frac{q_k}{2}>0.
  \end{equation}
  There are $(a^k_n)_n\subseteq\dZ^N$ and a strictly increasing
  function $\varphi_k\colon\dN\to\dN$ such that
  $\wlim_{n\to\infty} (-a^k_n)\star u^k_{1,\varphi_k(n)}= v^k$ in $H^1$.  By
  \eqref{eq:15} and by
  Theorem~\ref{thm:almost-bl-splitting}\ref{item:2} and \ref{item:3}
  there exists a sequence $(v^k_n)_n\subseteq H^1$ such that
  \begin{align}
    \label{eq:28}
    \lim_{n\to\infty}v^k_n
    &= v^k,&&\text{in $H^1$,}\\
    \wlim_{n\to\infty}(-a^k_n)\star u^k_{i,\varphi_k(n)}
    &=v^k,&&\text{in $H^1$, for } i=1,2,\label{eq:9}
  \end{align}
  and
  \begin{equation*}
    \lim_{n\to\infty}\bigabs{\mathcal{F}((-a^k_n)\star u^k_{i,\varphi_k(n)})
    -\mathcal{F}((-a^k_n)\star
    u^k_{i,\varphi_k(n)}-v^k_n)-\mathcal{F}(v^k)}_\nu=0,
  \qquad i=1,2.
  \end{equation*}

  Set $u^{k+1}_{i,n}\coloneqq u^k_{i,\varphi_k(n)}-a^k_n\star v^k_n$.  By
  the equivariance of $\cF$ and the invariance of the involved
  norms under the $\dZ^N$-action,
  \begin{align}
    \label{eq:14}
    \lim_{n\to\infty}\bigabs{\mathcal{F}(u^k_{i,\varphi_k(n)})-\mathcal{F}(u^{k+1}_{i,n})
      -\mathcal{F}(a^k_n\star v^k)}_\nu&=0,
    &&\text{for $i=1,2$,}\\
  \intertext{and, since by \eqref{eq:9} $\norm{\,\cdot\,}^2$ BL-splits along
    $(-a^k_n)\star u^k_{i,\varphi_k(n)}$ with respect to $v^k$,}
    \label{eq:16}
    \lim_{n\to\infty}\bigabs{\norm{u^k_{i,\varphi_k(n)}}^2-\norm{u^{k+1}_{i,n}}^2
      -\norm{v^k}^2}&=0,&&\text{for $i=1,2$.}
  \end{align}
  Equations \eqref{eq:16} and \eqref{eq:5} (for $k$) imply that
  \begin{equation*}
    \max_{i=1,2}\limsup_{n\to\infty}\norm{u^{k+1}_{i,n}}\le C_1,
  \end{equation*}
  hence \eqref{eq:5} for $k+1$.  The
  definition of the sequences $u^{k+1}_{i,n}$ and \eqref{eq:15}
  (for $k$) imply that
  \begin{equation}\label{eq:30}
    \lim_{n\to\infty}\abs{u^{k+1}_{1,n}-u^{k+1}_{2,n}}_p
    =\lim_{n\to\infty}\abs{u^{k}_{1,\varphi_k(n)}-u^{k}_{2,\varphi_k(n)}}_p
    =0,
  \end{equation}
  hence \eqref{eq:15} for $k+1$.  It follows from \eqref{eq:14}
  and \eqref{eq:11} (for $k$) that
  \begin{equation}\label{eq:31}
    \liminf_{n\to\infty}\abs{\cF(u^{k+1}_{1,n})-\cF(u^{k+1}_{2,n})}_\nu
    =\liminf_{n\to\infty}\abs{\cF(u^{k}_{1,\varphi_k(n)})-\cF(u^{k}_{2,\varphi_k(n)})}_\nu
    \ge C_2,
  \end{equation}
  hence \eqref{eq:11} for $k+1$.  Last but not least, from
  \eqref{eq:20} (for $k$), \eqref{eq:7}, and \eqref{eq:9} it
  follows that
  \begin{equation}
    \label{eq:21}
    \lim_{n\to\infty}\abs{a^m_{\psi^k_m(n)}-a^k_n}
    =\infty\qquad\text{if } m<k.
  \end{equation}
  Since \eqref{eq:29} is true for $k$, together with
  \eqref{eq:21} we obtain \eqref{eq:29} for $k+1$.  Moreover,
  \eqref{eq:21}, \eqref{eq:20} (for $k$) and \eqref{eq:28} yield
  \begin{multline*}
    \wlim_{n\to\infty}(-a^\ell_{\psi^k_\ell(n)})\star u^{k+1}_{i,n}
    =\wlim_{n\to\infty}\biglr(){(-a^\ell_{\psi^{k-1}_\ell(\varphi_k(n))})
      \star u^{k}_{i,\varphi_k(n)}-(a^k_n-a^\ell_{\psi^k_\ell(n)})\star v^k_n}\\
    = 0,
    \qquad\text{in $H^1$, if } \ell<k.
  \end{multline*}
  By the definition of $a^k_n$,
  \begin{equation*}
    \wlim_{n\to\infty}(-a^k_n)\star u^{k+1}_{i,n}
    =\wlim_{n\to\infty}\biglr(){(-a^k_n)\star u^{k}_{i,\varphi_k(n)}-v^k_n}
    =0,
    \qquad\text{in $H^1$.}
  \end{equation*}
  This proves \eqref{eq:20} for $k+1$.

  We now skip a finite number of elements of the sequences
  constructed in this induction step and adapt $\varphi_k$
  accordingly.  Choosing $m\in\dN$ large enough, by \eqref{eq:30}
  and \eqref{eq:31} we obtain
  \begin{align*}
    \abs{u^{k+1}_{1,m+n}-u^{k+1}_{2,m+n}}_p&\le\frac1{k+1}\\
    \shortintertext{and}
    \abs{\cF(u^{k+1}_{1,m+n})-\cF(u^{k+1}_{2,m+n})}_\nu
    &\ge C_2-\frac1{k+1}
  \end{align*}
  for all $n\in\dN$.  Property \eqref{eq:29} (for $k+1$) implies
  that
  \begin{equation*}
    \lim_{n\to\infty}\abs{a^m_{\psi^k_m(n)}-a^\ell_{\psi^k_\ell(n)}}
    =\lim_{n\to\infty}\abs{a^m_{\psi^\ell_m(\psi^k_\ell(n))}-a^\ell_{\psi^k_\ell(n)}}
    =\infty,
    \qquad\text{if }m<\ell\le k.
  \end{equation*}
  Since $\norm{\,\cdot\,}^2$ BL-splits along weakly convergent
  sequences this yields, together with \eqref{eq:28}, that
  \begin{equation*}
    \lim_{n\to\infty}\biggnorm{\sum_{j=\ell}^{k}
      a^j_{\psi^{k}_j(n)}\star v^j_{\psi^{k}_j(n)}}^2
    =\sum_{j=\ell}^{k}\norm{v^j}^2
  \end{equation*}
  for all $\ell\le k$.  For large enough $m$ this implies
  \begin{equation*}
    \biggnorm{\sum_{j=\ell}^{k}a^j_{\psi^{k-1}_j(\varphi_k(m+n))}
      \star v^j_{\psi^{k-1}_j(\varphi_k(m+n))}}^2
    \le2\sum_{j=\ell}^{k}\norm{v^j}^2,
    \qquad\text{for all } n\in\dN\text{ and }\ell\le k.
  \end{equation*}
  Fixing $m$ with these properties, writing $u^{k+1}_{i,n}$,
  $a^k_n$, and $v^k_n$ instead of $u^{k+1}_{i,m+n}$, $a^k_{m+n}$,
  and $v^k_{m+n}$, respectively, and writing $\varphi_k(n)$
  instead of $\varphi_k(m+n)$, all properties proved above remain
  valid, and, in addition, the following hold true:
  \begin{align}
    \abs{u^{k+1}_{1,n}-u^{k+1}_{2,n}}_p&\le\frac1{k+1}\label{eq:24}\\
    \shortintertext{and}
    \abs{\cF(u^{k+1}_{1,n})-\cF(u^{k+1}_{2,n})}_\nu
    &\ge C_2-\frac1{k+1}\label{eq:25}
  \end{align}  
  for all $n\in\dN$ and
  \begin{equation}
    \label{eq:18}
    \biggnorm{\sum_{j=\ell}^{k}a^j_{\psi^{k}_j(n)}\star v^j_{\psi^{k}_j(n)}}^2
    \le2\sum_{j=\ell}^{k}\norm{v^j}^2,
    \qquad\text{for all } n\in\dN\text{ and }\ell\le k.
  \end{equation}

  Now we consider the process of constructing sequences as
  finished and proceed to prove properties of the whole set.
  By induction, \eqref{eq:16} leads to
  \begin{equation*}
    \norm{u^{k+1}_{1,n}}^2
    =\norm{u^0_{1,\psi^k_{-1}(n)}}^2-\sum_{j=0}^k\norm{v^j}^2+o(1),
    \qquad\text{as }n\to\infty,
  \end{equation*}
  and hence $\sum_{j=0}^\infty\norm{v^j}^2\le C_1$ by
  \eqref{eq:5}.  In view of \eqref{eq:7} this yields
  \begin{equation}
    \label{eq:17}
    q_k\to0,\qquad\text{as }k\to\infty.
  \end{equation}

  We claim that the
  diagonal sequence $(u^n_{1,n})$ satisfies
  \begin{equation}
    \label{eq:19}
    b_n\star u^n_{1,n}\weakto0,\qquad\text{in $H^1$, as
      $n\to\infty$, for every sequence $(b_n)\subseteq\dZ$.}
  \end{equation}
  Note that by construction, for all $\ell\le k$
  \begin{equation*}
    u^k_{1,n}
    =u^\ell_{1,\psi^{k-1}_{\ell-1}(n)}
    -\sum_{j=\ell}^{k-1}a^j_{\psi^{k-1}_j(n)}\star v^j_{\psi^{k-1}_j(n)}.
  \end{equation*}
  Hence we have the representation
  \begin{equation}
    \label{eq:22}
    u^n_{1,n}
    =u^k_{1,\psi^{n-1}_{k-1}(n)}
    -\sum_{j=k}^{n-1}a^j_{\psi^{n-1}_j(n)}\star v^j_{\psi^{n-1}_j(n)},
    \qquad\text{if } n\ge k.
  \end{equation}
  First we show that
  \begin{equation}\label{eq:23}
    \wlim_{n\to\infty}(-a^k_{\psi^{n-1}_k(n)})\star u^n_{1,n}=0,
    \qquad\text{in $H^1$, for all $k\in\dN_0$.}
  \end{equation}
  Fix $k\in\dN_0$.  For every $w\in H^1$ and $\varepsilon>0$
  there is $\ell_0\ge k+1$ such that
  \begin{equation*}
    \norm{w}^2\sum_{j=\ell_0}^\infty\norm{v^j}^2\le\varepsilon^2/2.
  \end{equation*}
  Then \eqref{eq:18}, \eqref{eq:22}, and the translation
  invariance of the norm yield for $n\ge\ell_0$
  \begin{multline*}
    \bigabs{\bigscp{(-a^k_{\psi^{n-1}_k(n)})\star u^n_{1,n},w}}\\
    \begin{aligned}
      &\le\bigabs{\bigscp{(-a^k_{\psi^{n-1}_k(n)})\star u^{k+1}_{1,\psi^{n-1}_k(n)},w}}
      +\biggabs{\biggscp{\sum_{j=k+1}^{\ell_0-1}(a^j_{\psi^{n-1}_j(n)}-a^k_{\psi^{n-1}_k(n)})\star
          v^j_{\psi^{n-1}_j(n)},w}}\\
      &\hspace{22em}+\norm{w}\,\biggnorm{\sum_{j=\ell_0}^{n-1}a^j_{\psi^{n-1}_j(n)}\star
        v^j_{\psi^{n-1}_j(n)}}\\
      &\le\bigabs{\bigscp{(-a^k_{\psi^{n-1}_k(n)})\star u^{k+1}_{1,\psi^{n-1}_k(n)},w}}
      +\biggabs{\biggscp{\sum_{j=k+1}^{\ell_0-1}(a^j_{\psi^{n-1}_j(n)}-a^k_{\psi^{n-1}_k(n)})\star
          v^j_{\psi^{n-1}_j(n)},w}} +\varepsilon.
    \end{aligned}
  \end{multline*}
  It is easy to see that the sequence $(\psi^{n-1}_k(n))_n$ is
  strictly increasing.  Hence the first term in the last
  expression tends to $0$ as $n\to\infty$ by \eqref{eq:20}, and
  the second term tends to $0$ by \eqref{eq:28} and
  \eqref{eq:21}.  Since $\varepsilon>0$ and $w\in H^1$ were
  arbitrary, this proves \eqref{eq:23}.

  To finish the proof of \eqref{eq:19}, suppose for a
  contradiction that $\wlim_{n\to\infty} b_n\star u^n_{1,n}=
  v\neq0$ in $H^1$, for a subsequence.  Equation~\eqref{eq:23}
  implies that
  \begin{equation*}
    \lim_{n\to\infty} \xabs{b_n+a^k_{\psi^{n-1}_k(n)}} =\infty,
  \end{equation*}
  for every $k\in\dN_0$.  Pick $k\in\dN_0$ such that
  $q_{k}<\norm{v}$.  This is possible by \eqref{eq:17}.  Then,
  for every $w\in H^1$, it follows from \eqref{eq:18} and
  \eqref{eq:22} that
  \begin{equation*}
    \bigabs{\bigscp{b_n\star u^{k}_{1,\psi^{n-1}_{k-1}(n)}-v,w}}
    \le\bigabs{\bigscp{b_n\star u^{n}_{1,n}-v,w}}
    +\biggabs{\biggscp{\sum_{j=k}^{n-1}(b_n+a^j_{\psi^{n-1}_j(n)})
        \star v^j_{\psi^{n-1}_j(n)},w}}
    \to0
  \end{equation*}
  as $n\to\infty$, similarly as above.  Hence
  \begin{equation*}
    \wlim_{n\to\infty}\xlr(){b_n\star
      u^k_{1,\psi^{n-1}_{k-1}(n)}}
    = v
  \end{equation*}
  with $\norm{v}>q_k$.  Since
  $\biglr(){u^k_{1,\psi^{n-1}_{k-1}(n)}}_n$ is a subsequence of
  $(u^k_{1,n})_n$, this contradicts the definition of $q_k$ and
  proves \eqref{eq:19}.

  We are now in the position to finish the proof of uniform
  continuity of $\cF$.  Equations \eqref{eq:24} and \eqref{eq:25}
  imply that
  \begin{gather}
    \lim_{n\to\infty}\abs{u^n_{1,n}-u^n_{2,n}}_p=0\label{eq:32}\\
    \shortintertext{and}
    \liminf_{n\to\infty}\abs{\cF(u^n_{1,n})-\cF(u^n_{2,n})}_\nu\ge C_2.\label{eq:33}
  \end{gather}
  By Lemma~\ref{lem:lions} and \eqref{eq:19} $u^n_{1,n}\to0$ in
  $L^p$.  Together with \eqref{eq:32} and \eqref{eq:33} this
  contradicts the continuity of $\cF$ on $L^p$ and therefore
  proves the assertion about uniform continuity.
  
  It only remains to prove BL-splitting for $\cF$ along weakly
  convergent sequences in $H^1$ with respect to their weak
  limits.  Suppose that $u_n\weakto v$ in $H^1$.  By
  Theorem~\ref{thm:almost-bl-splitting}\ref{item:2} there is a
  sequence $(v_n)\subseteq H^1$ such that $v_n\to v$ in $H^1$
  and, after passing to a subsequence of $(u_n)$,
  \begin{equation}\label{eq:27}
    \cF(u_n)-\cF(u_n-v_n)\to\cF(v),
    \qquad\text{in } L^\nu
  \end{equation}
  as $n\to\infty$.  Since $(u_n)$ and $(v_n)$ are bounded in
  $H^1$, and by the uniform continuity of $\cF$ on bounded
  subsets of $H^1$ with respect to the $L^p$-norm (and hence also
  with respect to the $H^1$-norm), it follows that we may replace
  $v_n$ by $v$ in \eqref{eq:27}.  Using this, a standard
  reasoning by contradiction yields the claim.
\end{proof}

\section{Construction of Examples}
\label{sec:constr-exampl}

\begin{proof}[Proof of Example~\ref{exa:oscillation}]
  We first treat the case $f(t)\coloneqq\cos(\pi t)\abs{t}^p$.
  Set $R_n\coloneqq n^{-p/N}$ and fix a sequence
  $(x_n)\subseteq\dR^N$ such that $\abs{x_n}\to\infty$ and
  $B_{R_m}(x_m)\cap B_{R_n}(x_n)=\varnothing$ for $m\neq n$.
  Define real functions $u$ and $u_n$ on $\dR^N$ by setting
  \begin{equation*}
    u\coloneqq\sum_{k=1}^\infty\chi_{B_{R_k}(x_k)},\qquad
    w_n\coloneqq 2n\chi_{B_{R_n}(x_n)},\qquad\text{and}\qquad
    u_n\coloneqq u+w_n,
  \end{equation*}
  for each $n\in\dN$.  It is straightforward to show that
  $u\in L^p$, that $(u_n)$ is a bounded sequence in $L^p$, and
  that $u_n\to u$ pointwise and in $L^p_\rmloc$.  On the other
  hand, denoting by $\omega_N$ the volume of the unit ball in
  $\dR^N$, we obtain
  \begin{multline}\label{eq:34}
    \xabs{\int_{\dR^N}f(u_n)-f(u_n-u)-f(u)}\\
    \begin{aligned}
      &=\xabs{\int_{B_{R_n}(x_n)}f(u+w_n)-f(w_n)-f(u)}\\
      &=\xabs{\int_{B_{R_n}(x_n)}
        \cos((2n+1)\pi)(2n+1)^p
        -\cos(2n\pi)(2n)^p
        -\cos\pi}\\
      &=\xabs{\int_{B_{R_n}(x_n)}(-(2n+1)^p-(2n)^p+1)}\\
      &=\omega_N\xlr(){\xlr(){2+\frac1n}^p+2^p-\fracwithdelims(){1}{n}^p}\\
      &\to 2^{p+1}\omega_N,
    \end{aligned}
  \end{multline}
  as $n\to\infty$.  Since $2^{p+1}\omega_N>0$, this implies the
  claim.
  
  For the other example, $f(t)\coloneqq\cos(\pi/t)\abs{t}^p$, we
  set $R_n\coloneqq n^{p/N}$ and fix a sequence
  $(x_n)\subseteq\dR^N$ such that $\abs{x_n}/R_n\to\infty$ and
  $B_{R_m}(x_m)\cap B_{R_n}(x_n)=\varnothing$ for $m\neq n$.  We define
  \begin{equation*}
    u\coloneqq\sum_{k=1}^\infty\frac{1}{2n(2n-1)}\chi_{B_{R_k}(x_k)},\qquad
    w_n\coloneqq \frac{1}{2n}\chi_{B_{R_n}(x_n)},
    \qquad\text{and}\qquad
    u_n\coloneqq u+w_n
  \end{equation*}
  for each $n\in\dN$.  Then again, $u\in L^p$, $(u_n)$ is a
  bounded sequence in $L^p$, and $u_n\to u$ pointwise and in
  $L^p_\rmloc$.  For $x\in B_{R_n}(x_n)$ we obtain
  \begin{equation}\label{eq:35}
    u(x)+w_n(x)=\frac{1}{2n(2n-1)}+\frac1{2n}=\frac1{2n-1}
  \end{equation}
  and hence
  \begin{multline}\label{eq:36}
    \xabs{\int_{\dR^N}f(u_n)-f(u_n-u)-f(u)}\\
    \begin{aligned}
      &=\xabs{\int_{B_{R_n}(x_n)}f(u+w_n)-f(w_n)-f(u)}\\
      &\ge\xabs{\int_{B_{R_n}(x_n)}
        \cos((2n-1)\pi)\fracwithdelims(){1}{2n-1}^p
        -\cos(2n\pi)\fracwithdelims(){1}{2n}^p
        }-\int_{B_{R_n}(x_n)}\fracwithdelims(){1}{2n(2n-1)}^p\\
      &=\xabs{\int_{B_{R_n}(x_n)}-\fracwithdelims(){1}{2n-1}^p
        -\fracwithdelims(){1}{2n}^p}
      -\int_{B_{R_n}(x_n)}\fracwithdelims(){1}{2n(2n-1)}^p\\
      &=\omega_N\xlr(){\fracwithdelims(){1}{2-\frac1n}^p
          +\fracwithdelims(){1}{2}^p-\fracwithdelims(){1}{2(2n-1)}^p}\\
      &\to \frac{\omega_N}{2^{p-1}},
    \end{aligned}
  \end{multline}
  as $n\to\infty$.  This yields the claim.
\end{proof}
\begin{proof}[Proof of Remark~\ref{rem:counterexamples}]
  The construction of these counterexamples is closely related to
  Example~\ref{exa:oscillation}.  First consider the function
  $f(t)\coloneqq\cos(\pi/t)t^2$.  We define the
  Lipschitz-continuous cut-off function $\eta\colon\dR\to\dR$ by
  \begin{equation*}
    \eta(t)\coloneqq
    \begin{cases}
      1,&\qquad t\le0,\\
      1-t,&\qquad 0<t<1,\\
      0,&\qquad 1\le t,
    \end{cases}
  \end{equation*}
  introduce $R_n\coloneqq n^{2/N}$, pick a sequence
  $(x_n)\subseteq\dR^N$ such that $\abs{x_n}/R_n\to\infty$ and
  $B_{R_m+1}(x_m)\cap B_{R_n+1}(x_n)=\varnothing$ for $m\neq n$,
  and define
  \begin{equation*}
    u(x)\coloneqq\sum_{k=1}^\infty\frac{1}{2n(2n-1)}\eta(\abs{x-x_k}-R_k),\qquad
    w_n(x)\coloneqq \frac{1}{2n}\eta(\abs{x-x_n}-R_n),
  \end{equation*}
  and $u_n\coloneqq u+w_n$ for each $n\in\dN$ and $x\in\dR^N$.
  It is straightforward to check that $u,w_n\in H^1$ and that
  $(w_n)$ is bounded in $H^1$.  Since $w_n\to0$ a.e.,
  $w_n\weakto0$ in $H^1$.  Using \eqref{eq:35} we estimate
  \begin{multline*}
    \int_{B_{R_n+1}(x_n)\ssm B_{R_n}(x_n)}\bigabs{f(u+w_n)-f(w_n)-f(u)}\\
    \le\int_{B_{R_n+1}(x_n)\ssm B_{R_n}(x_n)}
    \xlr(){\fracwithdelims(){1}{2n-1}^2+\fracwithdelims(){1}{2n}^2
      +\fracwithdelims(){1}{2n(2n-1)}^2}\\
    \le\frac{3\omega_N}{n^2}((R_n+1)^N-R_n^N)
    =3\omega_N((1+n^{-2/N})^N-1)
    \to0
  \end{multline*}
  as $n\to\infty$.  Hence by the calculation in \eqref{eq:36}
  \begin{multline*}
    \xabs{\int_{\dR^N}f(u_n)-f(u_n-u)-f(u)}\\
    \begin{aligned}
      &=\xabs{\int_{B_{R_n+1}(x_n)}f(u+w_n)-f(w_n)-f(u)}\\
      &\ge\xabs{\int_{B_{R_n}(x_n)}f(u+w_n)-f(w_n)-f(u)}\\
      &\hspace{7em}-\int_{B_{R_n+1}(x_n)\ssm B_{R_n}(x_n)}\bigabs{f(u+w_n)-f(w_n)-f(u)}\\
      &\to \frac{\omega_N}{2^{p-1}}
    \end{aligned}
  \end{multline*}
  and the claim follows.  Note that the example above has no
  simple analogue in the case $f(t)\coloneqq\cos(\pi/t)\abs{t}^p$
  for $p>2$, using $R_n\coloneqq n^{p/N}$ as in the proof of the
  second case of Example~\ref{exa:oscillation}.  The reason is
  that the analogously defined sequence $(w_n)$ is not bounded in
  $L^2$ in that case.

  Now we treat the function
  $f(t)\coloneqq\cos(\pi t)\abs{t}^{2^*}$.  To this end put
  $R_n\coloneqq n^{-2^*/N}$, fix a sequence $(x_n)\subseteq\dR^N$
  such that $\abs{x_n}\to\infty$ and
  $B_{2R_m}(x_m)\cap B_{2R_n}(x_n)=\varnothing$ for $m\neq n$, and
  choose $\gamma\in(0,1)$ small enough such that
  \begin{equation}
    \label{eq:37}
    \biglr(){3^{2^*}+2^{2^*}+1}\biglr(){(1+\gamma)^N-1}
    \le \frac{2^{2^*+1}}{2}.
  \end{equation}
  Define
  \begin{equation*}
    u(x)\coloneqq\sum_{k=1}^\infty\eta\fracwithdelims(){\abs{x-x_k}-R_k}{\gamma
      R_k},\qquad
    w_n(x)\coloneqq 2n \eta\fracwithdelims(){\abs{x-x_k}-R_n}{\gamma
      R_n},
  \end{equation*}
  and $u_n\coloneqq u+w_n$ for each $n\in\dN$ and $x\in\dR^N$.
  It follows that $\supp(w_n)=\olB_{(1+\gamma)R_n}(x_n)$ for each
  $n$, where $\olB_r(z)$ denotes the closed ball in $\dR^N$ with
  radius $r$ and center $z$.  Again, it is straightforward to
  check that $u,w_n\in H^1$, that $(w_n)$ is bounded in $H^1$,
  and that $w_n\weakto0$ in $H^1$.  Using \eqref{eq:37} we
  estimate
  \begin{multline*}
    \int_{B_{(1+\gamma)R_n}(x_n)\ssm B_{R_n}(x_n)}\bigabs{f(u+w_n)-f(w_n)-f(u)}\\
    \begin{aligned}
      &\le\int_{B_{(1+\gamma)R_n}(x_n)\ssm B_{R_n}(x_n)}
      \biglr(){(1+2n)^{2^*}+(2n)^{2^*}+1}\\
      &\le\omega_N\biglr(){(3n)^{2^*}+(2n)^{2^*}+n^{2^*}}\biglr(){((1+\gamma)R_n)^N-R_n^N}\\
      &=\omega_N\biglr(){3^{2^*}+2^{2^*}+1}\biglr(){(1+\gamma)^N-1}\\
      &\le \frac{2^{2^*+1}\omega_N}{2}
    \end{aligned}
  \end{multline*}
  for all $n$.  Hence by the calculation in \eqref{eq:34}
  \begin{multline*}
    \xabs{\int_{\dR^N}f(u_n)-f(u_n-u)-f(u)}\\
    \begin{aligned}
      &=\xabs{\int_{B_{(1+\gamma)R_n}(x_n)}f(u+w_n)-f(w_n)-f(u)}\\
      &\ge\xabs{\int_{B_{R_n}(x_n)}f(u+w_n)-f(w_n)-f(u)}\\
      &\hspace{7em}-\int_{B_{(1+\gamma)R_n}(x_n)\ssm B_{R_n}(x_n)}\bigabs{f(u+w_n)-f(w_n)-f(u)}\\
      &\ge\xabs{\int_{B_{R_n}(x_n)}f(u+w_n)-f(w_n)-f(u)}
      -\frac{2^{2^*+1}\omega_N}{2}\\
      &\to \frac{2^{2^*+1}\omega_N}{2}
    \end{aligned}
  \end{multline*}
  and the claim follows.  Note that this example has no simple
  analogue in the case $f(t)\coloneqq\cos(\pi t)\abs{t}^p$ for
  $p<2^*$, using $R_n\coloneqq n^{-p/N}$ as in the proof of the
  first case of Example~\ref{exa:oscillation}.  Here reason is
  that for the analogously defined sequence $(w_n)$,
  $(\nabla w_n)$ is not bounded in $L^2$.
\end{proof}
\begin{acknowledgement}
  I would like to thank Kyril Tintarev for an informative
  exchange on this subject.  Moreover, I thank the referee for
  suggesting to analyze the limiting cases in
  Theorem~\ref{thm:main}.
\end{acknowledgement}
\bibliographystyle{amsplain-abbrv} \bibliography{unicobl}

\def\cprime{$'$} \def\polhk#1{\setbox0=\hbox{#1}{\ooalign{\hidewidth
  \lower1.5ex\hbox{`}\hidewidth\crcr\unhbox0}}}
\providecommand{\bysame}{\leavevmode\hbox to3em{\hrulefill}\thinspace}
\providecommand{\MR}{\relax\ifhmode\unskip\space\fi MR }
\providecommand{\MRhref}[2]{%
  \href{http://www.ams.org/mathscinet-getitem?mr=#1}{#2}
}
\providecommand{\href}[2]{#2}
\begin{thebibliography}{10}

\bibitem{MR2088936}
N.~Ackermann, \emph{On a periodic {S}chr{\"o}dinger equation with nonlocal
  superlinear part}, Math. Z. \textbf{248} (2004), no.~2, 423--443.
  \MR{MR2088936 (2005i:35075)}

\bibitem{MR2216902}
\bysame, \emph{A nonlinear superposition principle and multibump solutions of
  periodic {S}chr{\"o}dinger equations}, J. Funct. Anal. \textbf{234} (2006),
  no.~2, 277--320. \MR{MR2216902}

\bibitem{MR2151860}
N.~Ackermann and T.~Weth, \emph{Multibump solutions of nonlinear periodic
  {S}chr{\"o}dinger equations in a degenerate setting}, Commun. Contemp. Math.
  \textbf{7} (2005), no.~3, 269--298. \MR{MR2151860}

\bibitem{MR1066204}
J.~Appell and P.P. Zabrejko, \emph{Nonlinear superposition operators},
  Cambridge Tracts in Mathematics, vol.~95, Cambridge University Press,
  Cambridge, 1990. \MR{MR1066204 (91k:47168)}

\bibitem{MR3109657}
Y.V. Bogdanskii, \emph{Laplacian with respect to a measure on a {H}ilbert space
  and an {$L_2$}-version of the {D}irichlet problem for the {P}oisson
  equation}, Ukrainian Math. J. \textbf{63} (2012), no.~9, 1336--1348.
  \MR{3109657}

\bibitem{MR699419}
H.~Br{\'e}zis and E.~Lieb, \emph{A relation between pointwise convergence of
  functions and convergence of functionals}, Proc. Amer. Math. Soc. \textbf{88}
  (1983), no.~3, 486--490. \MR{699419 (84e:28003)}

\bibitem{MR1885519}
J.~Chabrowski, \emph{Weak convergence methods for semilinear elliptic
  equations}, World Scientific Publishing Co. Inc., River Edge, NJ, 1999.
  \MR{1885519 (2003a:35056)}

\bibitem{MR1985790}
G.~Da~Prato and J.~Zabczyk, \emph{Second order partial differential equations
  in {H}ilbert spaces}, London Mathematical Society Lecture Note Series, vol.
  293, Cambridge University Press, Cambridge, 2002. \MR{1985790}

\bibitem{MR3236753}
\bysame, \emph{Stochastic equations in infinite dimensions}, second ed.,
  Encyclopedia of Mathematics and its Applications, vol. 152, Cambridge
  University Press, Cambridge, 2014. \MR{3236753}

\bibitem{MR2321894}
Y.~Ding and A.~Szulkin, \emph{Bound states for semilinear {S}chr{\"o}dinger
  equations with sign-changing potential}, Calc. Var. Partial Differential
  Equations \textbf{29} (2007), no.~3, 397--419. \MR{MR2321894 (2008c:35067)}

\bibitem{MR2356423}
Y.~Ding and J.~Wei, \emph{Semiclassical states for nonlinear {S}chr\"odinger
  equations with sign-changing potentials}, J. Funct. Anal. \textbf{251}
  (2007), no.~2, 546--572. \MR{2356423}

\bibitem{MR0227747}
L.~Gross, \emph{Potential theory on {H}ilbert space}, J. Functional Analysis
  \textbf{1} (1967), 123--181. \MR{0227747}

\bibitem{MR2491945}
W.~Kryszewski and A.~Szulkin, \emph{Infinite-dimensional homology and multibump
  solutions}, J. Fixed Point Theory Appl. \textbf{5} (2009), no.~1, 1--35.
  \MR{MR2491945}

\bibitem{MR778970}
P.L. Lions, \emph{The concentration-compactness principle in the calculus of
  variations. {T}he locally compact case. {I}}, Ann. Inst. H. Poincar\'e Anal.
  Non Lin\'eaire \textbf{1} (1984), no.~2, 109--145. \MR{778970 (87e:49035a)}

\bibitem{MR778974}
\bysame, \emph{The concentration-compactness principle in the calculus of
  variations. {T}he locally compact case. {II}}, Ann. Inst. H. Poincar\'e Anal.
  Non Lin\'eaire \textbf{1} (1984), no.~4, 223--283. \MR{778974 (87e:49035b)}

\bibitem{MR88e:35170}
\bysame, \emph{Solutions of {H}artree-{F}ock equations for {C}oulomb systems},
  Comm. Math. Phys. \textbf{109} (1987), no.~1, 33--97. \MR{88e:35170}

\bibitem{MR1724248}
E.~Priola, \emph{On a class of {M}arkov type semigroups in spaces of uniformly
  continuous and bounded functions}, Studia Math. \textbf{136} (1999), no.~3,
  271--295. \MR{1724248}

\bibitem{MR2294665}
K.~Tintarev and K.H. Fieseler, \emph{Concentration compactness}, Imperial
  College Press, London, 2007, Functional-analytic grounds and applications.
  \MR{2294665 (2009b:46074)}

\bibitem{MR1400007}
M.~Willem, \emph{Minimax theorems}, Progress in Nonlinear Differential
  Equations and their Applications, 24, Birkh\"auser Boston, Inc., Boston, MA,
  1996. \MR{1400007 (97h:58037)}

\end{thebibliography}
\end{document}